\documentclass[11pt]{amsart}
\usepackage{amssymb}
\usepackage{amsmath}
\renewcommand{\baselinestretch}{1.2}

\newtheorem{theorem}{Theorem}
\newtheorem{lemma}{Lemma}
\newtheorem{corollary}{Corollary}
\newtheorem{proposition}{Proposition}

\theoremstyle{definition}

\newtheorem{remark}{Remark}
\newtheorem{example}{Example}



\title{$U$-Invariants for forms of higher degree}

\author{S. Pumpl\"un}
\email{susanne.pumpluen@nottingham.ac.uk}
\address{School of Mathematics\\
University of Nottingham\\
University Park\\
Nottingham NG7 2RD\\
United Kingdom
 }

\begin{document}

\maketitle

\begin{abstract} Both a general and a diagonal $u$-invariant
for forms of higher degree are defined, generalizing the $u$-invariant of quadratic forms.
Both old and new results on these invariants are collected.
\end{abstract}

\section*{Introduction}
Let $k$ be a field of characteristic $0$ or greater than $d$.
The {\it
$u$-invariant (of degree $d$)} of $k$ is defined as $u(d, k) =
\sup \{ \dim_k \varphi \}$, where $\varphi$ ranges over all anisotropic forms
of degree $d$ over $k$.
Since for $d \geq 3$ not all forms are diagonal, we also
define the {\it diagonal $u$-invariant (of degree $d$)} over $k$
 as $u_{diag}(d, k) = \sup \{ \dim \varphi \}$, where
$\varphi$ ranges over all anisotropic diagonal forms over $k$.
Obviously, $u_{diag}(d, k) \leq u (d, k)$. For $d=2$ and ${\rm char}\, k\not=2$, the definitions of
$u_{diag}(d, k)$ and $ u (d, k)$ coincide and correspond to the classical $u$-invariant
of quadratic forms over $k$.

 The $d^{ th}$
{\it level} (called {\it power Stufe} [P-A-R]) $s_d (k)$ of $k$ is the least
positive integer $s$ for which the equation
$$- 1 = a_1^d + \dots + a_s^d$$
is solvable in $k$. If there is no such integer, define $s_d (k) =\infty$. In case $d$ is odd, $s_d (k) = 1$.
 For $d = 2$ the level of a
field $s_2 (k)$ was studied among others by Pfister [Pf2], who proved that $s_2(k)$ is always a power of 2 if finite.

A look at the literature reveals that there indeed exist many results on the $u$-invariant and the $d^{ th}$ level:

Parnami-Agrawal-Rajwade [P-A-R] investigated $s_4 (k)$ for $k$ a finite
field ${\mathbb F}_q$ or an imaginary quadratic number field.
For instance, $s_4(\mathbb Q(\sqrt{-m}))=15$ for $m=7\,{\rm mod}\, 8$, and
$s_4(\mathbb Q(\sqrt{-2}))=6$. Thus for $d\geq 3$,
the $d^{th}$ level is not always a power of 2.
Higher levels have been studied in the literature mostly for finite fields.
The value of $s_d ({\mathbb F}_q)$
was computed by Pall and Rajwade [P-R] for $d \leq 10$, using cyclotomic
numbers and Jacobi sums for finite fields. Amice-Kahn [A-K] studied $s_d ({\mathbb F}_q)$ for $d$ a power of $2$.
More recently, the $d^{th}$ levels of the finite prime fields ${\mathbb F}_p$
were computed by Becker-Canales [B-C]. They proved that $s_d({\mathbb F}_p)$
is determined by a formula involving the coefficients of the Gauss period equation of degree $d$
associated to $p$. These results gave new insight in the behaviour of $d^{th}$ levels of $p$-adic fields
$\mathbb{Q}_p$, $p\geq 2$.

 Until Merkurjev's celebrated results very little
was known about the values  the $u$-invariant of quadratic forms can take for a given nonreal field.
 The situation for forms of higher degree is similar.

A well-known result of Brauer [G, (8.1)] states that the $u$-invariants
$u(2, k), u(3, k),\dots, u(d, k)$ must be finite, provided that the diagonal $u$-invariants
$u_{diag}(2, k), u_{diag}(3, k),\dots, u_{diag}(d, k)$ are finite.
Estimates for $u(d, k)$ using $u_{diag}(d, k)$ and $u(s, k)$ for $s=2,\dots,d-1$
or $u_{diag}(d, k)$ are given by Leep-Schmidt [Le-S, Theorem 2] and others.

Chevalley's well-known theorem that every finite field is a $C_1$-field [S, p.~97]
implies that $u(d,{\mathbb F}_q)\leq d$. M. Orzech [O] proved the bound $u_{diag}(d,{\mathbb F}_q)\leq d-1$,
if $-1\in {\mathbb F}_q^{\times d}$ and $d \geq 4$.
For $d \leq 5$, there exist anisotropic diagonal forms of degree $d$ and dimension 3
precisely if $d = 4$ and $q = 5, 13, 29$; or if $d = 5$ and $q = 11$ [O], which yields
 $u (d,{\mathbb F}_q) \geq 3$ in these cases.

 Let $k$ be a $p$-adic field (i.e. a finite field extension of ${\mathbb Q}_{ p}$, or of
$\mathbb{F}_q((x))$) with
residue class field of cardinality $q=p^r$. These fields attracted a lot of
attention due to Artin's conjecture on the $u$-invariant  which
 states that $u(d,{\mathbb Q}_{ p})\leq d^2$. This
was verified for $d = 2$ by Hasse, for $d = 3$ (independently by Lewis [L] and
Demyanov [De1]) and for $d = 5, 7$ and 11, under the assumption that $q$ is large
enough (Birch and Lewis [Bi-L], Laxton-Lewis [La-L], see also Knapp [Kn] and
Leep-Yeomans [Le-Y]). Lower bounds for the size of $q$ in these cases were given in [Kn], [Le-Y].
 (Artin's conjecture is believed to be true for forms of
prime degree, it does not hold in general.)

Joly [J2, p.~97] proved $u_{diag}(d,{\mathbb Q}_{ p}) \leq d^2 $ unless $p=2$ and $d=2^r$, and that if $p=2$ and $d=2^r$,
then
$u_{diag}(d,{\mathbb Q}_{ p}) \leq 2 d^2 $. If $d$ is not divisible by $p$ and $k$ is an extension field of
${\mathbb Q}_p$ of degree $n \geq 2$, this inequality also holds for $k$. Alemu [A] obtained some more estimates for
$u_{diag}(d, k)$ if $p$ divides $n$ and $n \geq 2$.

For every prime $p$, the form $\langle  1,p,\dots,p^{d-1}\rangle$ of degree $d$ over ${\mathbb Q}$
is anisotropic over ${\mathbb Q}$ [I-R, p.~150]. Hence $u_{diag} (d, {\mathbb Q}) \geq d$.
If $K$ is an algebraic number field of finite degree over $\mathbb Q$, and if $d$ is an odd positive integer,
then there is an integer $M(K,d)$ such that for $n>M(K,d)$, any form of degree $d$ in $n$ variables
over $K$ is isotropic [G].
Heath-Brown [H-B] showed that each nonsingular cubic form over ${\mathbb Q}$
of dimension $\geq 10$ is isotropic using the ``circle method'' which was
already applied by Davenport [D] to show that any (even degenerate) cubic form
over ${\mathbb Q}$ of dimension $\geq 16$ is isotropic.
Moreover, there exist nondegenerate cubic forms in $9$ variables which are
anisotropic (e.g., the norm form of any central simple division algebra of degree $3$ over ${\mathbb Q}$)
 implying that $u(3,{\mathbb Q})\geq 9$.

A field $k$ is called a $C_i$-field if every form of degree $d$ over $k$ in at least $d^i+1$ variables is isotropic,
that is we have $u(d,k)\leq d^i$ in this case.
 If $k$ is a $C_i$-field, then so is each algebraic field extension $l$ over $k$, hence
also $u(d,l)\leq d^{i}$. The rational function field in one variable $k(t)$ is a $C_{i+1}$-field.
Moreover, $k = {\mathbb C} (t_1, \dots t_n)$ is a $C_n$-field.
These results of Tsen and Lang motivated much of the later work done in this direction ([S, p.~97], cf. also [G]).

 The techniques used to prove the above results are all intricate and
sophisticated and have their origin in different areas of mathematics.
In this paper more information about these ``higher'' invariants is collected and obtained.
 In the first section we cite some well-known or easy-to-prove statements for further reference.

In the second section we use Kneser's upper bound for the $u$-invariant of  forms of higher degree over fields with
finite square class group (Leep [Le, 2.3]) to obtain upper bounds for the diagonal $u$-invariant of $\mathbb{Q}_p$ or
finite field extensions of $\mathbb{Q}_p$, which superceed the ones given by Alemu [A].

In section three we generalize another classical result from the theory of quadratic forms
on the behaviour of the diagonal $u$-invariant of Henselian valued fields:
 let $(k, v)$ be a Henselian valued
field whose residue field $\overline{k}$ has char $\overline{k} \nmid d$.
Let $\Gamma$ be the value group of $v$. Then
$$u_{diag}(d, k)=|\Gamma/d \Gamma| u_{diag}(d, \overline{k})$$
with the obvious conventions for computing with infinite values. For quadratic forms
this result is due to Springer. As applications we get a wide range of  diagonal $u$-invariants:
Without using Tsen-Lang theory
we show $u_{diag}(d, K)=d^n u_{diag}(d, k)$ for the iterated
power series field $K=k((x_1))\dots((x_n))$ over any field with ${\rm char}\, k \nmid d$.
Let $k$ be a $p$-adic field (e.g. a finite field extension of $\mathbb{Q}_p$, or of
$\mathbb{F}_q((x))$) such that ${\rm char}\, \overline{k} \nmid d$.
Then $u_{diag}(d, k)=d \, u_{diag}(d, \overline{k})$.
 Let $R$ be a discrete valuation ring with residue field $\overline{k}$ and field of fractions
$k$. If $\overline{k}$ is algebraically closed then $k$ is $C_1$, i.e. $u(d, k)\leq d$. This was proved by Lang
using the theory of Witt vectors and Witt polynomials [G, (6.25)]. Using our generalization of Springer's theorem
for forms of degree $d$
we show that the bound $u(d, k)\leq d$
is indeed best possible for  algebraically closed $\overline{k}$ if
${\rm char}\, \overline{k} \nmid d$.

 It is a long-standing question whether the finiteness of
the classical $u$-invariant $u(k)$ (of quadratic forms) of a field $k$ of characteristic not 2 implies the finiteness of $u(k(t))$.
For $k$ a non-dyadic $p$-adic field
this was proved by Hoffmann-van Geel [Ho-V] and independently by Merkurjev [M].
Parimala-Suresh [Pa-S] were able to lower the bound for the $u$-invariant of a function field $K$ of transcendence degree one
 over a non-dyadic $p$-adic field from $u(K)\leq 22$ [Ho-V] to $u(K)\leq 10$. It is easy to see that
 $u(K)\geq 8$ and it was conjectured that $u(K)=8$.
It was shown recently by Zahidi [Z] that any quadratic form over $\mathbb{Q}(t_1,\dots,t_n)$
of dimension greater than $2^{n+2}$ will be
isotropic over the field $\mathbb{Q}_p(t_1,\dots,t_n)$ for almost all primes $p$, using the Ax-Kochen-Ersov transfer
theorem from the model theory of valued fields [Ax-K]. In particular,
any quadratic form over $\mathbb{Q}(t)$ of dimension greater than $8$ is
isotropic over $\mathbb{Q}_p(t)$ for almost all primes $p$. This result resembles two others,
also proved with the help of the Ax-Kochen-Ersov Theorem. The first result is due to Greenleaf [G, (9.2)]. It
states that every polynomial with integer coefficients without constant term of degree $d$ in $n>d$ variables
is isotropic over $\mathbb{Q}_p$ for almost all primes $p$. The second one is due to Ax-Kochen [G, (7.4)]. It states
that given a degree $d$, the field $\mathbb{Q}_p$ is a $C_2(d)$-field [G] for almost all primes $p$.
 I.e.,
for almost all primes $p$ a form of degree $d$ over $\mathbb{Q}_p$ of dimension greater than or equal to $d^2+1$
will be isotropic.

Applying Zahidi's method [Z] to
forms of higher degree, we conclude that any form of degree $d$ over the rational function field
$\mathbb{Q}(t_1,\dots,t_n)$ of dimension greater than $d^{n+2}$ is isotropic over
$\mathbb{Q}_p(t_1,\dots,t_n)$ for all but finitely many primes $p$.

\section{Preliminaries}

\subsection{}
Let $k$ be a field of characteristic $0$ or greater than $d$.
A $d$-{\it linear form} over $k$ is a $k$-multilinear map $\theta : V \times
\dots \times V \to k$ ($d$-copies) on a finite-dimensional vector space $V$ over $k$ which is {\it symmetric}, i.e. $\theta (v_1,
\dots, v_d)$ is invariant under all permutations of its variables.
A {\it form of degree $d$} over $k$ is a map $\varphi:V\to k$ on a finite-dimensional vector space $V$ over $k$
 such that $\varphi(a v)=a^d\varphi(v)$ for all $a\in k$, $v\in V$ and such that the map $\theta : V \times
\dots \times V \to k$ defined by
 $$\theta(v_1,\dots,v_d)=\frac{1}{d!}\sum_{1\leq i_1< \dots<i_l\leq d}(-1)^{d-l}\varphi(v_{i_1}+ \dots +v_{i_l})$$
($1\leq l\leq d$) is a $d$-linear form over $k$.
By fixing a basis $\{e_1,\dots,e_n\}$ of $V$, any form $\varphi$ of degree $d$ can be viewed as a homogeneous polynomial
of degree $d$ in $n={\rm dim}\, V$ variables $x_1,\dots,x_n$
 via $\varphi(x_1,\dots,x_n)=\varphi(x_1e_1+\dots+x_ne_n)$ and,
vice versa, any homogeneous polynomial of degree $d$ in $n$ variables over $k$ is a form of degree $d$ and dimension $n$
over $k$. Any $d$-linear form $\theta : V \times
\dots \times V \to k$ induces a form $\varphi: V\to k$ of degree $d$ via
$\varphi (v)=\theta(v,\dots,v)$.
 We can identify $d$-linear forms and forms of degree $d$ with
the help of the obvious correspondence.

\subsection{} Two $d$-linear spaces $(V_i,\theta_i)$, $i=1,2$, are called {\it isomorphic} (written
$(V_1,\theta_1)\cong (V_2,\theta_2)$ or just $\theta_1\cong\theta_2$) if there exists a bijective linear map
$f:V_1\to V_2$ such that $\theta_2(f(v_1),\dots,f(v_d))=\theta_1(v_1,\dots,v_d)$ for all $v_1,\dots,v_d\in V_1.$
A $d$-linear space $(V,\theta)$ (or the $d$-linear form $\theta$) is
called {\it nondegenerate} if $v = 0$ is the only vector such
that $\theta (v, v_{2}, \dots, v_d) = 0$ for all $v_i \in V$. A form of degree $d$ is
called {\it nondegenerate} if its associated $d$-linear form
is nondegenerate.

\subsection{}
The {\it orthogonal sum} $(V_1,\theta_1)\perp (V_2,\theta_2)$ of two $d$-linear spaces $(V_i,\theta_i)$, $i=1,2$, is defined
to be the $k$-vector space $V_1\oplus V_2$ together with the $d$-linear form
$$(\theta_1 \perp\theta_2)(u_1+v_1,\dots,u_d+v_d)=\theta_1(u_1,\dots,u_d)+\theta_2(v_1,\dots,v_d)$$
($u_i\in V_1$, $v_i\in V_2$). The {\it tensor product} $(V_1,\theta_1)\otimes (V_2,\theta_2)$
 is the $k$-vector space $V_1\otimes V_2$ together with the $d$-linear form
$$(\theta_1 \otimes \theta_2)(u_1\otimes v_1,\dots,u_d\otimes v_d)=\theta_1(u_1,\dots,u_d)\cdot\theta_2
(v_1,\dots,v_d)$$ [H-P].

A $d$-linear space $(V,\theta)$ is
called {\it decomposable} if $(V,\theta)\cong (V,\theta_1)\perp (V,\theta_2)$
for two non-zero $d$-linear spaces $(V,\theta_i)$, $i=1,2$.
A non-zero $d$-linear space $(V,\theta)$ is
called {\it indecomposable} if it is not decomposable and {\it absolutely indecomposable}, if
it stays indecomposable under each algebraic field extension.
If $d\geq 3$, $a_i,b_j\in k^\times$, then
$$\langle  a_1,\ldots ,a_n\rangle\cong \langle  b_1,\ldots ,b_n\rangle$$
 if and only if there is a permutation $\pi\in S_n$ such that
 $\langle b_i \rangle\cong \langle a_{\pi (i)}\rangle$ for every $i$.
(This is a special case of Harrison's Krull-Schmidt Theorem for $d$-linear forms [H, 2.3].)

\section{Some estimates on higher $u$-invariants}

Let $\varphi$ be a form of degree $d$ on a $k$-vector space
$V$. Write $D(\varphi)= \{a \in k^\times\mid\varphi(x)=a$ for some $x\in V\}$ for the set of non-zero
elements represented by $\varphi$. $\varphi$ is called {\it universal} if
$D(\varphi)=k^\times$.

\begin{lemma} ([Pu], cf. [L, p.~14] for $d=2$) Let $\varphi$ be a (perhaps degenerate) form of
degree $d$ over $k$ and $a\in k^\times$.\\
(i) If $a\in D(\varphi)$ then $\varphi\perp \langle  -a\rangle  $ is isotropic.\\
(ii) If $\varphi$ is anisotropic and $\varphi\perp \langle  -a\rangle  $ isotropic then $a\in D(\varphi)$.
\end{lemma}

Similar to the well-known case $d = 2$, the $u$-invariants $u(d, k)$ and $u_{diag}(d, k)$ can be
characterized as follows:

\begin{lemma} (i) The diagonal $u$-invariant
$u_{diag}(d, k)$ is the smallest integer $n$ such that all diagonal forms of degree $d$ over
$k$ of dimension greater than $n$ are isotropic,
 and the $u$-invariant $u(d, k)$ is the smallest integer $n$ such that all forms of degree $d$ over
$k$ of dimension greater than $n$ are isotropic.\\
(ii) If $u = u(d, k)$ then each anisotropic form of degree $d$ over $k$ of
dimension $u$ is universal. If $u = u_{diag}(d, k)$ then each diagonal anisotropic form of degree $d$ over $k$ of
dimension $u$ is universal.\\
(iii) $u_{diag}(d, k) \leq \min \{ n \, \vert \,{\rm all}\,{\rm forms}\,{\rm of}\,{\rm degree}\, d \,{\rm over}\, k \, {\rm of}\,
  {\rm dimension}\, \geq n \,{\rm are}$

 \noindent $\,{\rm universal}\}$ with the understanding that the ``minimum'' of an empty set of integers
  is the symbol $\infty$.
\end{lemma}

The proof is trivial and used Lemma 1.

\begin{lemma} (i)  $u(3,k)\not=2$.\\
(ii) For every integer $n$ we have $u(d^n,k)\geq u(d,k)^n$.\\
(iii) If $u(d_1,k)=u$ then $u(md_1,k)\geq u$ for each integer $m>1$.
\end{lemma}

\begin{proof} (i) Suppose that $u(3,k)=2$.
 This means there exists a homogeneous polynomial $f(x,y)=a_0x^3+a_1x^{2}y+a_2xy^2+a_3y^3$ in two variables over $k$
 which is anisotropic. Therefore there exists
an irreducible polynomial $f(t)=a_0+a_1t+a_2t^2+a_3t^3$  of degree $3$ over $k$, because
$f(x,y)$ is isotropic over $k$ if and only if $f(t)$ has a root in $k$.
(Note that $f(a,b)=a^3f(b/a)$ for $a\not=0$ and that $f(0,b)=a_3 b^3$.)
Hence there is a field extension $l/k$ of degree $3$. Let $n_{l/k}$ be its anisotropic norm
and let $\{v_1,v_2,v_3\}$ be a $k$-basis of
$l$. Then the form $g(x_1,x_2,x_3)=n_{l/l}(x_1v_1+x_2v_2+x_3v_3)$  is anisotropic and thus $u(3,k)\geq 3$, a contradiction.\\
(ii) Let $f(X)=f(x_1,\dots,x_u)$ be an anisotropic form of degree $d$ in $u$ variables over $k$. Let $f_2(X_1,\dots,X_u)=
f(f(X_1),\dots,f(X_u))$ where each $X_j$ is a different set of $u$ variables.
Since $f_2$  is an anisotropic form of degree $d^2$ in $u^2$ variables,
we have $u(d^2,k)\geq u^2$. Repeating this argument yields the assertion (see also [S, p.~99, 15.7]).\\
(iii) Suppose that $\varphi_1$ is an anisotropic form of degree $d_1$ over $k$. Then
$\varphi$ defined via $\varphi(z_1,\dots,z_u)=\varphi_1(z_1,\dots,z_u)^m$ is anisotropic of degree $md_1$.
\end{proof}

\begin{remark} (i) We have
$s_d (k) \leq u_{diag}(d, k) \leq u (d, k).$
This generalizes another well-known result for quadratic forms.
\\
(ii) Let $k$ be algebraically closed, then $\vert k^{\times} /k^{\times d}
\vert = 1$ and each form of degree $d$ and dimension $> 1$ over $k$ is
isotropic. In particular, $u_{diag}(d, k)=u(d, k) = 1$.
On the other hand, let $k$ be formally real.
Since $-1\not\in \sum k^2$, also $-1\not\in \sum k^d$ for even $d$. Thus $m\times \langle  1\rangle$
is an anisotropic form of degree $d$ for each integer $m$ and for $d$ even, we always know that
$u_{diag}(d,k)=\infty$,  hence $u(d,k)=\infty$.
Clearly, $u(d,\mathbb{R})=1$ for odd integers $d$.
\end{remark}

\begin{example}
(communicated to the author by D. Leep) Since the field $\mathbb{F}_p$ has an algebraic extension of degree $d$, there exists a homogeneous form
$\varphi(x_1,\dots,x_d)\in \mathbb{Z}[x_1,\dots,x_d]$ of degree $d$ in $d$ variables which is anisotropic modulo $p$.
The form $f=\langle 1,p,\dots, p^{d-1}\rangle\otimes \varphi$ is
of degree $d$ in $d^2$ variables with coefficients in $\mathbb{Z}$. Since $\varphi$ is anisotropic modulo $p$,
$f$ must be anisotropic over $\mathbb{Q}_p$. It follows that $f$ is anisotropic over $\mathbb{Q}$.
Thus $$u(d,{\mathbb Q}_p)\geq d^2$$ and $$u(d,{\mathbb Q})\geq d^2.$$
\end{example}

\begin{lemma} Let $\varphi$ be a form of degree $d$ over $k$.\\
(i) If $\varphi$ is
anisotropic over $k$, then $\varphi$ remains anisotropic over $k(t)$.\\
(ii) (for $d = 2$, cf. [La, p.~256]). If $\varphi$ is an anisotropic
form of degree $d$ over $k$, then
$$D(\varphi_{k(t)}) \cap k = D(\varphi).$$
\end{lemma}

\begin{proof} (i) Suppose that $\varphi$ is isotropic over $k(t)$, then there
are $f_i (t)/g_i (t) \in k(t)$, not all of them zero, such that $\varphi
\left( \frac{f_1(t)}{g_1(t)} , \dots, \frac{f_n(t)}{g_n(t)} \right) = 0$.
Clearing denominators we can assume without loss of generality that there are $f_i (t) \in
k[t]$, not all of them zero, such that $\varphi (f_1 (t), \dots, f_n (t)) =
0$. Changing these if necessary assume moreover that $t$ does not divide all
of them. Put $t = 0$ to obtain $\varphi (f_1 (0), \dots, f_n (0)) = 0$ with
not all of the $f_i (0)$ zero, so we found an isotropic vector of $\varphi$
over $k$.\\
(ii) Let $a \in D (\varphi_{k(t)}) \cap k$, then $\varphi_{k(t)} \perp
\langle - a \rangle$ is isotropic over $k(t)$ and therefore so is $\varphi
\perp \langle - a \rangle$ over $k$ by (i). This implies $a \in D (\varphi)$ by Lemma 1.
\end{proof}

 Let $R$ be a commutative unital ring, let $t$ be an
indeterminate over $R$ and let $S$ be any of the rings $R(t), R((t)), R[t]$
or $R[[t]]$. Given two homogeneous polynomials $f, g$ of degree $d$ over $R$,
the homogeneous polynomial $f \perp tg$ of degree $d$ over $S$ is anisotropic if
and only if both $f$ and $g$ are anisotropic over $R$ [Hu-R].

\begin{lemma} Let $K= k(t)$ be the rational function field or let $K=k((t))$ be the Laurent series field over $k$
 and let $\varphi$ be a form of degree $d$ over $k$. Then the form
$$\langle 1,t,t^2,\dots,t^{d-1}\rangle\otimes \varphi=\varphi\perp
t\varphi \perp t^2 \varphi \perp \dots\perp t^{d-1} \varphi$$
 of degree $d$ is anisotropic over $K$ if and only if $\varphi$ is anisotropic over $k$.
\end{lemma}

\begin{proof}  Let $K= k(t)$ be the rational function field.
Let $\varphi$ be a form over $k$ in $n$ variables. Assume that the form $\varphi\perp
t\varphi \perp t^2 \varphi \perp \dots\perp t^{d-1} \varphi$ of degree $d$ is isotropic over $k(t)$.
Then there exist polynomials $f_i(t)\in k[t]$, not all of them zero, such that
$$ (1) \,\, \varphi(f_1(t),\dots,f_n(t))+
t\varphi(f_{n+1}(t),\dots,f_{2n}(t)) + t^2 \varphi(f_{2n+1}(t),\dots,f_{3n}(t)) + \dots$$
$$+ t^{d-1}
 \varphi(f_{(d-1)n+1}(t),\dots,f_{dn}(t))=0.$$
Assume additionally that the value for $\sum_{i=1}^n {\rm deg}\, f_i(t)$ is minimal. Plugging in $t=0$ shows that
$ \varphi(f_1(0),\dots,f_n(0))=0$ and thus
$f_i(0)=0$ for all $i$, $1\leq i\leq n$, because $ \varphi$ is anisotropic over $k$. Hence $f_i(t)=tg_i(t)$ for $1\leq i\leq n$. Substituting this into $(1)$ and cancelling $t$ yields
a version of $(1)$ with decreased $\sum_{i=1}^n {\rm deg}\, f_i(t)$, a contradiction.

The same argument applies if $K= k((t))$ is the Laurent series field.
\end{proof}

Indeed, using the above notation a similar argument shows that given forms $\varphi_i$ of degree $d$ over $k$,  the form
$\varphi_1\perp
t\varphi_2 \perp t^2 \varphi_3 \perp \dots\perp t^{d-1} \varphi_d$
 of degree $d$ is anisotropic over $K$ if and only if the forms $\varphi_i$ are  anisotropic over $k$ for all $i$.

\begin{corollary} Let $K= k(t)$ be the rational function field or let $K=k((t))$ be the Laurent series field over $k$.
Then
$$ u (d,K)\geq d\, u(d, k),$$
$$ u_{diag} (d,K)\geq d\, u_{diag}(d, k) ,$$
 and $s_d (k) = s_d (k(t))$.
\end{corollary}

\begin{proof} This is a consequence of Lemma 4 and Lemma 5.
\end{proof}

\begin{example} (i) Let $k = k_0 (t_1, \dots, t_n)$ with $k_0$ a field of
characteristic 0 or $> d$ and with $t_1, \dots, t_n$ independent
indeterminates over $k$. By induction on $n$ (using the above results) the form $\langle 1, t_1, t_1^2, \dots, t_1^{d-1} \rangle
\otimes \dots \otimes \langle 1, t_n, \dots, t_n^{d-1} \rangle$ of degree $d$ and
dimension $d^n$ is anisotropic over $k$. Hence $u(d, k_0
(t_1, \dots, t_n)) \geq u_{diag}(d, k_0 (t_1, \dots, t_n)) \geq d^n$. Since
$k = {\mathbb C} (t_1, \dots t_n)$ is a $C_n$-field [S,
p.~97] it follows that $u(d, {\mathbb C} (t_1, \dots, t_n)) \leq d^n$ and
$u_{diag}(d, {\mathbb C} (t_1, \dots, t_n)) \leq d^n$. We conclude that
$$u(d,{\mathbb C} (t_1, \dots, t_n)) = u_{diag}(d, {\mathbb C} (t_1, \dots, t_n)) =
d^n.$$
 Thus every power of $d$ is the $u(d, k)$-invariant (resp. the $u_{diag}(d,
k)$-invariant) of some suitable field $k$.\\
(ii) Let $k_0$ be a field of characteristic 0 or $> 3$, such that there exists a division algebra of degree 3 over $k_0$.
Then $u(3,k_0)\geq 9$ and  $u(3, k_0(t)) \geq 3\cdot 9=27$ by Proposition 1.
Indeed, there is an Albert division algebra over $k_0(t)$ with $t$ an indeterminate over $k_0$ [KMRT, p.~531].
Since its norm is anisotropic, it is an example of an absolutely indecomposable cubic form
over $k_0(t)$ of dimension $27$.
\end{example}

\section{Kneser's theorem for forms of higher degree}

\begin{theorem} (Kneser's theorem Leep [Le, 2.3] or [De2], see [S, p.~104] for $d=2$)
Let $k$ be a field of arbitrary characteristic with finite $d^{th}$ level $s_d
(k)  < \infty$. Then
$$u_{diag}(d,k) \leq \vert k^{\times} /k^{\times d} \vert.$$
In particular, let $\varphi\cong \langle  a_1,\dots,a_m\rangle$ be an anisotropic form of degree $d$ over $k$ such that
$m=|k^{\times} / k^{\times d}|$. Then $\varphi$ is universal or isotropic.
\end{theorem}

\begin{remark} Although Theorem 1 was already proved by Demyanov [De2] in 1956, it seems
that this paper, written in Russian and apparently never translated, is mostly unknown to the public.
 It was proved again for diagonal forms of degree $d$ over nonreal fields as well as for odd degree $d$
 over real fields  by Leep [Le, 2.3].
 Kneser's theorem for diagonal forms of degree $d$ over finite fields $\mathbb{F}_q$ can be found in
  [J1, Th\'{e}or\`{e}me 1, p.~25] or Small [Sm2, 3.13]).
The proofs presented in those papers are basically identical to the one employed in [De2] (1956).
\end{remark}

\begin{example} (i) Let $k=\mathbb{F}_q$ be a finite field, where $q=p^s$ with $p$ prime.
 Then
$|\mathbb{F}_q^{\times}/\mathbb{F}_q^{\times d}|= {\rm gcd}(d, q-1)$. Thus
$$u_{diag}(d,\mathbb{F}_q) \leq  {\rm gcd}(d, q-1).$$
In particular, $s_d(\mathbb{F}_q)=u_{diag}(d,\mathbb{F}_q)=1$, if $d$ is relatively prime to $q-1$.\\
(ii) Let $K$ be a finite field extension of the field $\mathbb{Q}_p$ of $p$-adic numbers. Since the $d^{th}$ power level
of $K$ is finite [G, (7.18)], we have
$$u_{diag}(d,K) \leq \frac{d}{|d|_p} w $$
 where $w$ is the number of $d^{th}$ roots of
unity contained in $K$ and $|d|_p=1/p^{ord_p(d)}$ with $ord_p(d)$ being the highest power of $p$ dividing $d$ [Ko, p.~73].
This greatly improves the bound given by Alemu [A] which depends on the
 degree of the field extension $n=[K:\mathbb{Q}_p]$:
$$u_{diag}(d,K) < {\rm max}\,(3nd^2-nd+1,2d^3-d^2)$$
if $p>2$ divides $d$, and
$$u_{diag}(d,K) < 4nd^2-nd+1$$
if $p=2$.

If $d$ is  not divisible by $p$ and $k$ contains no $d^{th}$ roots of unity other than $1$ itself
(e.g. $K=\mathbb{Q}_p$ and
$d$ is relatively prime to both $p$ and $p-1$ [Ko, p.~73]), then even $u_{diag}(d,K) \leq d$.
\end{example}

Indeed, there is a similar result as Theorem 1 for complete discrete valuation ring of characteristic 0
 with finite residue field (with the obvious definition for $u_{diag}(d,R)$):

\begin{proposition} ([G, p.~135]) Let $R$ be a complete discrete valuation ring
of characteristic 0 with finite residue field.\\
(i)  Suppose $d$ is odd. Then
  $$u_{diag}(d,R)\leq d \, |R^{\times}/R^{\times d}|.$$
 In particular, also $s_d(R)\leq d\, |R^{\times}/R^{\times d}|$.
\\
(ii)  Suppose $d$ is even. Then
  $$u_{diag}(d,R)\leq (1+m_d)\, d \, |R^{\times}/R^{\times d}|,$$
  where $m_d$ is the smallest positive integer such that $-m_d$ has a $d^{th}$ root in $R$.
 In particular, also $s_d(R)\leq  (1+m_d)\,d\, |R^{\times}/R^{\times d}|$.
\end{proposition}

 That $R$ has finite level is also shown in [J2].

\begin{remark} If $R$ is an integral domain with quotient field $K$, then
 $u_{diag}(d,R)=u_{diag}(d,K)$. For a complete discrete valuation ring $R$ of characteristic 0 with finite residue
 field,  we thus obtain
 $$u_{diag}(d,R)=u_{diag}(d,K)\leq {\rm min}( \vert K^{\times} /K^{\times d} \vert, d|R^{\times}/R^{\times d}|)$$
 if  $d$ is odd, and
  $$u_{diag}(d,R)=u_{diag}(d,K)\leq {\rm min}( \vert K^{\times} /K^{\times d} \vert, (1+m_d) d|R^{\times}/R^{\times d}|)$$
 if  $d$ is even.
\end{remark}

\begin{corollary} Let $K$ be a finite field extension of $\mathbb{Q}_p$ of degree $n$, with valuation ring $R$
and with residue field $\mathbb{F}_q$, where $q=p^f$. Then
  $$u_{diag}(d,R)\leq d\, {\rm gcd}(d,q-1)| \mathbb{Z}_p /d \mathbb{Z}_p|^e,$$
 if $d$ is odd and
   $$u_{diag}(d,R)\leq (1+m_d)\,d\, {\rm gcd}(d,q-1)| \mathbb{Z}_p /d \mathbb{Z}_p|^e,$$
if $d$ is even, with $e$ copies of $\mathbb{Z}_p /d \mathbb{Z}_p$,
  where $e$ is the ramification index of $K$ over $\mathbb{Q}_p$.
\end{corollary}

\begin{proof} We have $R^{\times}\cong \mathbb{F}_q^{\times} \times (1+\pi R)$, and
$1+\pi R \cong  \mathbb{Z}_p \oplus\dots\oplus  \mathbb{Z}_p$ with $e$ copies of $ \mathbb{Z}_p$. Hence
$$|R^{\times}/R^{\times d}|=
|\mathbb{F}_q^{\times}/\mathbb{F}_q^{\times d}|| \mathbb{Z}_p /d \mathbb{Z}_p \oplus\dots\oplus\mathbb{Z}_p /d \mathbb{Z}_p|,$$
with $e$ copies of $\mathbb{Z}_p /d \mathbb{Z}_p.$ It is well-known that
$|\mathbb{F}_q^{\times}/\mathbb{F}_q^{\times d}|= {\rm gcd}(d, q-1)$.
\end{proof}

In particular, $$u_{diag}(d,\mathbb{Z}_p)\leq d\, {\rm gcd}(d,p-1)| \mathbb{Z}_p /d \mathbb{Z}_p |,$$
 if $d$ is odd, and
 $$u_{diag}(d,\mathbb{Z}_p)\leq d\, {\rm gcd}(d,p-1)| \mathbb{Z}_p /d \mathbb{Z}_p |$$
 if $d$ is even.

\section{A theorem of Springer for higher $u$-invariants of Henselian valued fields}

\begin{theorem} (see [Sp2] for $d = 2$) Let $(k, v)$ be a discretely valued
field with valuation ring $ \mathcal{O}$, value group $\Gamma$ and residue field $\overline{k}$. Assume
 char $\overline{k} \nmid d$.
\\
(i)  $u_{diag}(d, k)\geq |\Gamma /d \Gamma| u_{diag}(d, \overline{k})$.
\\
 (ii) If $(k, v)$ is a Henselian valued field then
$$u_{diag}(d, k)=|\Gamma/d \Gamma| u_{diag}(d, \overline{k}).$$
The above (in)equalities still hold when the values are infinite.
\end{theorem}

For $d=2$, (ii) corresponds to Springer's Theorem for quadratic forms over Henselian valued fields.
For a similar result for quadratic forms, see also [Du].

\begin{proof} Choose a set $\{ \pi_{\gamma} \; \vert
\; \gamma \in I \}$ of representatives of the distinct cosets in $\Gamma
/d \Gamma$.\\
(i) Let $u=u(d, \overline{k})$. Choose an anisotropic form $\varphi$ over $\overline{k}$.  Lift $\varphi$ to a form $\tilde{\varphi}$ over $k$. The form $\phi = \bigoplus
\pi_{\gamma} \tilde{\varphi}$ is anisotropic over $k$, since it is anisotropic over
the completion of $k$ by Springer's Theorem [Mo].
\\ (ii) It remains to check that $u_{diag}(d, k) \leq |\Gamma /d \Gamma | \, u_{diag} (\overline{k})$.
Let $\varphi $ be an anisotropic
diagonal form of degree $d$ over $k$.  Decompose the $d$-linear form $\phi$ associated to $\varphi$ as $\phi = \bigoplus
\pi_{\gamma} \phi_{\gamma}$ with each $\phi_{\gamma}$ a unit form (a $d$-linear form $\theta$ with coefficients
in $\mathcal{O}$ is called a {\it unit form}
if the $d$-linear form $\bar \theta$ over $\bar k$ is nondegenerate  [Mo]). Let $m$ be the number
of those finitely many $\phi_{\gamma}$. Since $\varphi$ is anisotropic, all nonzero
$\overline{\phi_{\gamma}}$ must be anisotropic over  $ \overline{k} $ [Mo].
Therefore we conclude that $u(d, \overline{k}) \geq \dim
\overline{\phi}_{\gamma}$ for all nonzero $\overline{\phi}_{\gamma}$. Since
$\overline{\phi}_{\gamma}$ is also
a diagonal form, we have moreover that $u_{diag}(d,\overline{k}) \geq \dim
\overline{\phi_{\gamma}}$ for all nonzero $\overline{\phi}_{\gamma}$. Hence
if $u_{diag}(d, k) = \dim \varphi$, we conclude that $u_{diag}(d, k) = \sum \dim
\overline{\phi}_{\gamma} \leq m\, u (d, \overline{k})$ (resp. $\leq m\, u_{diag}(d,
\overline{k})$), where $m \leq n$ is the number of the $\phi_{\gamma}$
which are nonzero. Indeed, we have $m\leq |\Gamma/d \Gamma|$.
Thus $u_{diag}(d, k) \leq |\Gamma /d \Gamma | \, u_{diag} (\overline{k})$.
If $\Gamma /d \Gamma$ is an infinite group, we can take arbitrarily many $\pi_\gamma$ to obtain
an anisotropic form of arbitrarily large dimension by the construction above, which shows that
$u_{diag}(d, k)$ is infinite. A similar argument using a unit form which is the lift of an anisotropic form over
$\overline{k}$ shows that $u_{diag}(d, k)$ is infinite when $u_{diag}(d, \overline{k})$ is infinite.
\end{proof}

\begin{corollary} Let $k$ be a field of char $k \nmid d$.
Let $l/k$ be a field extension of finite type over $k$ of transcendence degree $n$. Then
$$u_{diag}(d, l)\geq d^n u_{diag}(d, k^{\prime})$$
for a suitable finite field extension $k^{\prime}/k$.
\end{corollary}

Note that for a Henselian valuation ring $R$ with residue field $\overline{k}$ and quotient field
$k$, $s_d(R)=s_d(k)$ [J2, (6.8)] and $s_d(R)=s_d(\overline{k})$ [R, 1.2] if ${\rm char}\,\overline{k}\nmid d$
 (use Hensel's Lemma).

 Let $R$ be a complete discrete valuation ring with residue field $\overline{k}$ and field of fractions
$k$. If $\overline{k}$ is algebraically closed then $k$ is $C_1$. This was proved by Lang
using the theory of Witt vectors and Witt polynomials [G, (6.25)]. Using our generalization of Springer's Theorem
 above and $u_{diag}(d, \overline{k})=1$, we calculate the (diagonal) $u$-invariant of $k$ in case
${\rm char}\, \overline{k} \nmid d$:

\begin{corollary}
(i)  Let $R$ be a complete discrete valuation ring with residue field $\overline{k}$ and field of fractions
$k$.
If $\overline{k}$ is algebraically closed and ${\rm char}\, \overline{k} \nmid d$, then
$$u_{diag}(d, k)=u(d, k)=d.$$
\noindent (ii)  Let $R$ be a complete discrete valuation ring with residue field $\overline{k}$ and field of fractions
$k$. Suppose that $R$ is unramified and that $\overline{k}$ is the algebraic closure of the field $\mathbb{F}_p$.
 Let $d$ be an integer such that $p \nmid d$. Then
 $$u_{diag}(d, k)= u (d, k)=d.$$
 (iii) Let $k_\infty$ be  the maximal unramified algebraic
 extension of $\mathbb{Q}_p$ and $d$ an integer such that $p \nmid d$. Then
 $$u_{diag}(d, k_\infty)= u (d, k_\infty)=d.$$
\end{corollary}

\begin{proof} (i) and (ii) are obvious.\\
(iii) The fact that $u_{diag}(d, k)= d$ follows directly from (i). The valuation ring of
$k_\infty$ is a Henselian discrete valuation ring [G, p.~99], hence also $u_{diag}(d, k_\infty)= d$ by Theorem 2.
It is well-known that both $k$ and $k_\infty$ are $C_1$ [G, (6.26)]. Thus
$u_{diag}(d, k_\infty)\leq u (d, k_\infty)\leq d$ and $u_{diag}(d, k)\leq u (d, k)\leq d$
which implies the assertion.
\end{proof}

Note that in case $R$ has the same characteristic as $\overline{k}$, $R$ must be isomorphic to the power series ring
$\overline{k}[[x]]$, in which case (i) is a special case of (ii).

\begin{remark} According to MathSciNet review MR0037836 (12,315d) of [De1], Demyanov proved that
 $u(3,k)\leq 3u(3,\overline{k})$ for any field $k$ which is complete under a
discrete valuation with residue class field $\overline{k}$ of characteristic not $3$.
This result is also proved by Springer [Sp1], this time including the case that ${\rm char}\,\overline{k}$ is 3.
This could be seen as an indication for the existence of a Springer Theorem for cubic forms (or
even of forms of  degree $d$ greater than $3$)
over Henselian valued fields which are not necessarily diagonal. However, Morandi [Mo, 2.7] showed that a general Springer
Theorem does not hold by  giving an example of an isotropic $d$-linear form $\bar f$ over $\mathbb{Q}$ which has
an anisotropic lift $f$ to $\mathbb{Q}((t))$.

In case of a finite residue class field, it was moreover shown that any cubic form of dimension greater than 3
over a field $k$
which is complete under a discrete valuation is isotropic in the unramified cubic extension $l$ of $k$
[Sp1, Proposition 3].
\end{remark}

\begin{remark} Let $k$ be a $C_i$-field (or even a $C_i^p$-field) and let $K=k((x))$ be
the field of formal Laurent series in $x$ over $k$. Then $K$ is a $C_{i+1}$-field [Pf1, Chapter 5, 2.2]
(or even a $C_{i+1}^p$-field as defined in Pfister [Pf1, Chapter 5]).

As observed in [Pf1, p.~111] for $d=2$ we can deduce from our generalization of Springer's
 Theorem even without Tsen-Lang theory that for instance
the iterated power series field $K=k((x_1))\dots((x_n))$ over  a field $k$ of char $k \nmid d$, $1\leq n \leq \infty$,
has $u_{diag}(d, K)= d^n u_{diag}(d, k)$.
In particular if $k$ is algebraically closed, we have $u_{diag}(d, K)=d^n$.
\end{remark}

\section{Some $u$-invariants of finite and $p$-adic fields}

\begin{remark}
(i) If $k$ is a $p$-adic field (e.g. a finite field extension of $\mathbb{Q}_p$, or of
$\mathbb{F}_q((x))$) such that ${\rm char}\, \overline{k} \nmid d$, then by Springer's Theorem
$$u_{diag}(d, k)=d \, u_{diag}(d, \overline{k}).$$
 \\(ii) We know that
$$u(d,\mathbb{Q}_p)\geq d^2$$
by Example 1. If $p \equiv 1 \,{\rm mod}\, 4$, then $-1\in \mathbb{Q}_p^{\times 2}$, therefore $s_d( \mathbb{Q}_p)=1$
for any integer of the type $d=4m+2$.
\end{remark}

 Let $k$ be a finite field extension of $\mathbb{Q}_p$ with residue class field
$\overline{k}=\mathbb{F}_q$. (Note that $s_d(R)=s_d(k)$ [J2] and that
$s_d(k)=s_d(\mathbb{F}_q)$ for $p\nmid d$ [R].)
 We now assume that ${\rm char}\, \mathbb{F}_q=p \nmid d$ to be able to apply Springer's Theorem for forms of higher degree,
which yields $u_{diag}(d, k)=d \, u_{diag}(d, \mathbb{F}_q)$.

 Let $d^*={\rm gcd}(d,q-1)$. Then $|\mathbb{F}_q^{\times}/\mathbb{F}_q^{\times d}|=
|\mathbb{F}_q^{\times}/\mathbb{F}_q^{\times d^*}|$ so that $u_{diag}(d,\mathbb{F}_q)=u_{diag}(d^*,\mathbb{F}_q)$.
 Kneser's Theorem for forms of higher degree together with the above yields
$$d\leq u_{diag}(d, k)=d \, u_{diag}(d, \mathbb{F}_q) \leq d d^*.$$

In particular, the inequality
$$u_{diag}(d, k)=d \, u_{diag}(d, \mathbb{F}_q)\leq d^2$$
is recovered if ${\rm char}\, \mathbb{F}_q=p \nmid d$.
This has been known for some time, see [A].

On the other hand, it also follows that
$$(1)\,\,\, d\,s_d(\mathbb{F}_q)\leq u_{diag}(d, k)=d \, u_{diag}(d, \mathbb{F}_q) \leq d d^*.$$
For odd integers $d$ we get in particular
$$(2)\,\,\, d \leq u_{diag}(d, k)=d \, u_{diag}(d, \mathbb{F}_q) \leq d d^*.$$

If $d$ is relatively prime to both $p$ and $q-1$, then
$$d\leq u_{diag}(d, k)=d u_{diag}(d, \mathbb{F}_q) \leq d d^*= d.$$
This yields
$$ (1')\,\,\, d\,s_d(\mathbb{F}_q)\leq u_{diag}(d, k)=d u_{diag}(d, \mathbb{F}_q) = d$$
 and proves

\begin{proposition} If $d$ is relatively prime to both $p$ and $q-1$, then
$$s_d(k)\leq u_{diag}(d, k)=d.$$
\end{proposition}

As far as the author knows, this result has not been previously mentioned in the literature. Some special
cases of this result are for instance the following:

\begin{example} (i) If $d$ is an odd integer which is not a multiple of $3$, then
$u_{diag}(d,\mathbb{Q}_3)=d.$
\\(ii) If $d$ is an odd integer which is not a multiple of $5$, then
$u_{diag}(d,\mathbb{Q}_5)=d.$
\\(iii) If $d$ is an odd integer which is not a multiple of $7$ or of $3$, then
$u_{diag}(d,\mathbb{Q}_7)=d.$
\end{example}

If $d\geq 4$ and $-1\in \mathbb{F}_q$ then $u_{diag}(d, \mathbb{F}_q)\leq d-1$ by [O], hence
$$u_{diag}(d, k)=d \,u_{diag}(d, \mathbb{F}_q) \leq d (d-1)=d^2-d.$$
Moreover, $s_{p-1}(\mathbb{F}_p)=p-1$ for any
$p\not=2$ [Ti] which implies
$u_{diag}(p-1, \mathbb{F}_p)=p-1$, since
$p-1\leq u_{diag}(p-1, \mathbb{F}_p)\leq {\rm gcd}(p-1,p-1)=p-1$. Thus
$$(3)\,\,\, u_{diag}(p-1, \mathbb{Q}_p)=(p-1)^2.$$
So the upper bound $u_{diag}(d, \mathbb{Q}_p)\leq d^2$ given by Joly [J2, p.~97] for any
$p\not=2$ is best possible. Furthermore,
$$(p-1)^2 \leq u(p-1, \mathbb{Q}_p).$$
Using this, the results of [P-A-R] on the $d^{th}$ level of a finite field for $d=4,6,8$ and $10$,
as well as the generalization of Kneser's Theorem to forms of higher degree
we compute higher $u$-invariants of finite and $p$-adic fields for even $d$:

\begin{example}
(i) $ u_{diag}(4, \mathbb{F}_5)=4$ by [P-A-R] or [PR], hence $u_{diag}(4, \mathbb{Q}_5)=16$. The latter result
follows also from (3) above. By Chevalley, it is clear that $ u(4, \mathbb{F}_5)=4$.\\
(ii) $ u_{diag}(4, \mathbb{F}_{25})\in \{ 3,4\}$, hence $u_{diag}(4, \mathbb{Q}_{25})\in \{ 12,16\}$.\\
(iii)
$s_4 (\mathbb{F}_{29})=3 \leq u_{diag}(4, \mathbb{F}_{29})\leq {\rm gcd}(4,28)=4$ [B-C, p.~434], hence
$ u_{diag}(4, \mathbb{F}_{29})\in \{ 3,4\}$, hence $u_{diag}(4, \mathbb{Q}_{29})\in \{ 12,16\}$.\\
(v) $ u_{diag}(4, \mathbb{F}_7)=2$, hence $u_{diag}(4, \mathbb{Q}_5)=8$.\\
(iv) $ u_{diag}(6, \mathbb{F}_7)=6$, hence $u_{diag}(6, \mathbb{Q}_7)=36$.
By Chevalley, we get also $ u(6, \mathbb{F}_7)=6$.\\
(vi) If $q$ is odd and $q \equiv 5 \, {\rm mod}\, 6$, then $ u_{diag}(6, \mathbb{F}_{q})=2$. In particular,
$ u_{diag}(6, \mathbb{F}_{11})=2$, hence $u_{diag}(6, \mathbb{Q}_{11})=12$.\\
(vii) If $p\in \{31, 67, 79, 139, 223\}$, then $ u_{diag}(6, \mathbb{F}_p)\in \{3,4,5,6 \}$  [B-C, p.~434], hence
$u_{diag}(6, \mathbb{Q}_p)\in \{18,24,30,36 \}$.
\end{example}

\begin{remark} Since
$s_2(k)\leq s_4(k)\leq \dots\leq s_{2^r}(k)$
for any field $k$, it follows that $s_4(\mathbb{F}_{29})=3\leq s_8(\mathbb{F}_{29})\leq {\rm gcd}\, (8,28)=4$,
that is $s_8(\mathbb{F}_{29})\in \{3,4\}$. Therefore, there is an error in the calculation of
$s_8(\mathbb{F}_{29})$ in [P-R]. Another error in the same paper was corrected in [B-C], who proved that
$s_6(\mathbb{F}_{31})=4$ and not equal to $3$ as claimed in [P-R].
\end{remark}

 Let $k_d$ be the set of elements of $k$ which are sums of $d^{th}$ powers. Any element in a finite field $\mathbb{F}_q$ which is a sum of $d^{th}$ powers must be a sum of
$d$ $d^{th}$ powers (Tornheim's Theorem [Sm2, 3.16]). As soon as $q$ is ``large enough'' with respect to the exponent
$d$, every element in $\mathbb{F}_q$ is a sum of two $d^{th}$ powers [Sm2, 6.12]. For instance,
every element in $\mathbb{F}_q$ is a sum of two $4^{th}$ powers provided $q>41$. For $q>(d^*-1)^2$
every element of $\mathbb{F}_q$ is a sum of two $d^{th}$ powers [Sm2, p.~148]. The proofs given in [Sm2] show
that the form $n\times\langle  1\rangle  = \langle  1,\dots,1\rangle  $ becomes universal for some $n$ under the given assumptions.
Indeed, if $q > (d^*-1)^4$  then $u_{diag}(d, \mathbb{F}_q)=2$ [Sm1].

Let $k$ be an arbitrary field. It is clear that every element in $k$ can be written as a sum of $n$ $d^{th}$ powers provided
the form $n\times \langle  1\rangle  =\langle  1,\dots,1\rangle  $ of degree $d$ is universal for some $n$. For a field $k$ of finite $d^{th}$ level
$s=s_d(k)$, the form $s\times \langle  1\rangle  =\langle 1,\dots,1\rangle  $ of degree $d$ is anisotropic over $k$.
This immediately leads to the following result:

\begin{proposition} Let $k$ be a field such that $s=s_d(k)=u_{diag}(d,k)$. Then
every element in $k$ can be written as a sum of $s$ $d^{th}$ powers.
\end{proposition}

\section{Isotropy of forms of higher degree over $p$-adic rational function fields }

Let $d$ be a fixed exponent.
We have both
 $$u_{diag}(d,\mathbb{Q}_p(t_1,\dots,t_n))\geq d^n \, u_{diag}(d,\mathbb{Q}_p)$$
and
$$u(d,\mathbb{Q}_p(t_1,\dots,t_n))\geq d^n \, u(d,\mathbb{Q}_p)$$
 (Proposition 1).
Choose $\mathbb{Q}_p$ such that $p \nmid d$, then $u_{diag}(d,\mathbb{Q}_p)=d\, u_{diag}(d,\mathbb{F}_p)$
by Springer's Theorem. Put $m_p= u_{diag}(d,\mathbb{F}_p)$, then
$$u_{diag}(d,\mathbb{Q}_p(t_1,\dots,t_n)) \geq d^{n+1}\, u_{diag}(d,\mathbb{F}_p)=d^{n+1} m_p$$
 for all $p$ with $p \nmid d$ (i.e for almost all primes).
Analogous to the case where $d=2$, one may conjecture that for all primes $p$ with $p \nmid d$,
$$u_{diag}(d,\mathbb{Q}_p(t_1,\dots,t_n))=d^{n+1} m_p.$$
Note that $ u_{diag}(d,\mathbb{F}_p)\leq d$ and $ u(d,\mathbb{F}_p)\leq d$ by Chevalley [S], therefore
$u_{diag}(d,\mathbb{Q}_p(t_1,\dots,t_n))=d^{n+1} m_p$ may take any of the following values: $d^{n+1},\,
2d^{n+1},\dots,(d-1)d^{n+1},\, d^{n+2}$. However, since $p \nmid d$ there is
some additional restriction on the possible values of $m_p$ because of $m_p\leq {\rm gcd}(d,p-1)$.

Let $K=k(t_1,\dots,t_n)$, $k$ a field. For any integer $s$, define $M_s(K)$ to be the set consisting of all (not necessarily nondegenerate)
forms $\varphi$ of degree $d$ over $K$ such that $\varphi$ is isometric to a form of degree $d$
 whose coefficients are
polynomials in $k[t_1,\dots,t_n]$ whose total degrees are bounded by s. For each form $\varphi$
over $K$ there exists an integer $s$ such that $\varphi \in  M_s(K)$.
Put
$$\mu_s(K) = {\rm sup}\{{\rm dim}(q) | q \in M_s(K) \text{ and $\varphi$ is anisotropic over K}\}.$$
Then $\mu_s(K)\leq \mu_{s+1}(K)$ and
${\rm sup}\,\mu_s(K)=u(d,K)$.

We are going to prove the following generalization of [Z, Corollary]:

\begin{theorem} Let $\varphi$ be a form of degree $d$ over $\mathbb{Q}(t_1,\dots,t_n)$ of dimension greater than
$d^{n+2}$. Then $\varphi$ is isotropic over the field $\mathbb{Q}_p(t_1,\dots,t_n)$ for almost all primes $p$.
\end{theorem}

This follows directly from the generalization of Zahidi's Main Theorem [Z]:

\begin{theorem} For any positive integers $n$, $s$ there exists a finite set of primes $P(n, s)$
such that $\mu_s(\mathbb{Q}_p(t_1,\dots,t_n))  \leq d^{n+2}$ for all primes $p$ with $ p\not\in  P(n, s)$.
\end{theorem}

As observed in [Z] for $d=2$, the theorem does not imply that the $u$-invariant of a
$p$-adic function field is finite for almost all primes $p$.
We have
$$u_{diag}(d,\mathbb{F}_p((t_1))\dots((t_n)))= d^n \, u_{diag}(d,\mathbb{F}_p)=d^n m_p$$
 by Springer's Theorem for all $p$ with $p \nmid d$ and
$$u(d,\mathbb{F}_p((t_1))\dots((t_n)))\leq d^{n+1} $$
 by Tsen-Lang theory for
all $p$ [G, (4.8)]. Since $u_{diag}(d,\mathbb{F}_p)\leq d$ by Chevalley,
$$u_{diag}(d,\mathbb{F}_p((t_1))\dots((t_n)))=d^n m_p\leq d^{n+1} $$
for
all $p$ with $p \nmid d$.

\begin{lemma}(for $d=2$, see [Z, 3.1]) Let $n,m$ and $s$ be integers greater than zero and let $m > d^{n+2}$. Then there
exists a positive integer $b(n, s)$ such that for any prime $p$ with $p \nmid d$ the following holds:
every form of degree $d$ and of dimension $m$ which is contained in
$M_s(\mathbb{F}_p((t))(t_1,\dots,t_n))$ has an
isotropic vector $v = (v_1, \dots, v_m)$ such that $v_i\in \mathbb{F}_p((t))[t_1,\dots,t_n]$ and ${\rm deg}(v_i) \leq b(n, d)$.
\end{lemma}

 The proof is analogous to the one given for $d=2$ by Zahidi.
  Given a form of degree $d$
over $\mathbb{F}_p((t))(t_1,\dots,t_n)$ this allows us to bound the degree of the entries of an isotropic vector as described
above in terms of the degrees of the coefficients of the given form.

\begin{lemma}(for $d=2$, see [Z, 3.2])
 Let $m = (m_1,m_2,m_3,m_4)$ be a quadruple of positive integers. There exists a first
order sentence $\Delta_m$ in the language of fields such that for every field $k$
 the following holds:
 every form $\varphi$ of degree $d$ in $M_{m_2}(k(x_1, \dots, x_{m_1}))$ of dimension $m_3$
has an isotropic vector of degree bounded by $m_4$ if and only if $k \models \Delta_m $.
\end{lemma}

\begin{proof}
 We sketch how to obtain the formula $\Delta_m$.
Let $\varphi(a_1,\dots,a_{s};y_1,\dots,y_{m_3})$ denote the homogeneous polynomial of degree $d$
in the variables $y_1,\dots,y_{m_3}$ with coefficients $a_1,\dots,a_{s}$.
 The statement that any $m_3$-dimensional form $\varphi$ of degree $d$ with coefficients bounded in
degree by $m_2$ has an isotropic vector with entries of degrees bounded by $m_4$ is true if and only if the
following statement is true in $k[t_1,\dots,t_n]$:
$ \forall a_1, \dots, a_{s} \exists y_1, \dots, y_{m_3} :$
$$(\bigvee_i a_i=0)\bigvee ((\bigwedge_i {\rm deg}(a_i)\leq m_2)\wedge (\bigwedge_i){\rm deg}(y_i)\leq m_4)\wedge
(\bigvee_i y_i\not=0)\wedge $$
$$(\varphi(a_1,\dots,a_{s};y_1,\dots,y_{m_3})=0)).$$
Since the degrees of the polynomials over which we quantify are bounded, we can replace these quantifiers which range over
the polynomial ring, by quantifiers ranging over the ground field $k$. As in the proof of the last lemma, the
quantifier-free part of the formula which still involves the indeterminates $t_1,\dots,t_{m_2}$, can be easily replaced by
a quantifier-free formula over the base field $k$. This yields the formula $\Delta_m$.
\end{proof}

We can now prove Theorem 5:

\begin{proof} Fix $n$ and $s$. For any prime $p$ with $p\nmid d$, $\mathbb{F}_p((t))$ satisfies the sentence
$\Psi=\Delta_m$ with $m=(n,s,d^{n+2}+1,b(n,s))$. By the Ax-Kochen-Theorem there exists a finite set of primes $P(n,s)$,
such that for all primes $p\not\in P(n,s)$, $\mathbb{Q}_p$ satisfies $\Psi$. Hence any form of degree $d$ in
$M_s(\mathbb{Q}_p(t_1,\dots,t_n))$, $p\not\in P(n,s)$, of dimension greater than or equal to $d^{n+2}+1$ is isotropic.
We obtain  $\mu_s(\mathbb{Q}_p(t_1,\dots,t_n))\leq d^{n+2}$ for all but finitely many primes.
\end{proof}

\begin{corollary} For any positive integers $n$, $s$ such that $s>n$, there exists a finite set of primes $P(n, s)$
such that
$\mu_s(\mathbb{Q}_p(t_1,\dots,t_n)) = d^{n+2}$, for all primes $p$ such that both $p \not\in P(n, s)$
and $u(d,\mathbb{Q}_p)=d^2$.
\end{corollary}

The proof is analogous to the one for $d=2$ given in [Z].

\bigskip\noindent
{\it Acknowledgements:}
 Parts of this paper were written during the author's stay  at the University of Trento
  supported by  the
``Georg-Thieme-Ged\"{a}chtnisstiftung'' (Deutsche Forschungsgemeinschaft).
  The author thanks K. Becher and P. Morandi
 for several useful conversations on the subject,  K. Zahidi for providing his preprint,
 and is greatly indebted to D. Leep, whose ideas and comments on an earlier draft greatly helped to improve the expositions.

\renewcommand{\baselinestretch}{1.0}

\end{document}